\documentclass{article}

\usepackage{amsfonts}
\usepackage{amsrefs}
\usepackage{mathrsfs,amsmath}
\usepackage{amssymb}
\usepackage{graphicx}
\usepackage{wrapfig}
\usepackage{geometry}
\usepackage{amsthm}
\usepackage{commath}
\usepackage{color}
\newtheorem{theorem}{Theorem}[section]

\newtheorem{lemma}[theorem]{Lemma}
\newtheorem{prop}{Proposition}[section]
\newtheorem{remark}{Remark}
\newtheorem*{remark*}{Remark}
\newtheorem*{theorem*}{Theorem}

\newtheorem{definition}{Definition}[section]

\begin{document}

\title{Farthest Point Map on a Centrally Symmetric Convex Polyhedron}
%

%
\author{Zili Wang}
%

%
%

\maketitle              

\begin{abstract}
The farthest point map sends a point in a compact metric space to the set of points farthest from it. We focus on the case when this metric space is a convex centrally symmetric polyhedron, so that we can compose the farthest point map with the antipodal map. The purpose of this work is to study the properties of their composition. We show that: 1. the map has no generalized periodic points; 2. its limit set coincides with its generalized fixed point set; 3. each of its orbit converges; 4. its limit point set is contained in a finite union of hyperbolas. We will define some of these terminologies later.

\end{abstract}

\section{Introduction}

On a compact metric space $\Sigma$, one can define the farthest point map $F$ as follows: for any $p\in \Sigma$, $F(p)$ is the set of all points $q$ such that the distance from $p$ is maximized at $q$. 

As an example, if $\Sigma$ is a sphere, then for any $p\in \Sigma$, $F(p)=\{\phi(p)\}$, where $\phi(p)$ is the antipodal point of $p$ on $\Sigma$. Then we say $F$ is single-valued ($F(p)$ has one element for any $p$), and also an involution, since $F(\phi(p))=\{p\}$. The ``converse'' to this statement is a conjecture by Steinhaus: if $\Sigma$ is convex, and $F$ is single-valued and involutive, then $\Sigma$ is a sphere. 

The conjecture was disproved by C. Vilcu in 2000  through the construction of a family of counter-examples (see \cite{V1}). But it led to a series of research work on the properties of the farthest point map $F$, especially in the context when $\Sigma$ is a convex surface. For instance, in \cite{Z}, T. Zamfiresu proved that $F$ is single-valued for all $p\in\Sigma$ except for a $\sigma$-porous set. With the additional assumption that $\Sigma$ is a polyhedral surface, J. Rouyer showed in \cite{R1} that $F$ is piecewise single-valued, and the multi-valued set is contained in a finite union of algebraic curves of degree at most 10. The interested reader may also refer to \cite{V2} for a good survey on this topic.

In this work, we are interested in the case $\Sigma$ is the surface of a centrally symmetric convex polyhedron equipped with the intrinsic path metric, where we observed some good properties in the dynamics of the farthest point map. 

We first introduce some notations and definitions before stating the main results. Let $\phi$ be the antipodal map on $\Sigma$. Define $f=F \circ \phi$.

\begin{definition}[Generalized Periodic Point and Fixed Point] Suppose there is a positive integer $n$ such that $p\in f^n(p)$ and $p\notin f^m(p)$ if $m<n$. We say p is a generalized periodic point of $f$ with order $n$ if $n>1$, and a generalized fixed point of $f$ if $n=1$. The latter case happens if and only if $\phi(p)\in F(p)$ (that is, $\phi(p)$ is a farthest point from $p$).

In the same way, we can define the generalized periodic point of $F$ with order $n$ ($F$ has no generalized fixed point). 

\end{definition}
This definition coincides with the usual definition of periodic and fixed points wherever $F$ and $f$ are single-valued.

\begin{definition} [Orbit]
A sequence $p_0, p_1, p_2, p_3 \dots$ is an orbit of $F$ if $p_n\in F(p_{n-1})$ for all $n\geq 1$. Similarly, we can define an orbit of $f$.

\end{definition}

If $p$ is a generalized periodic point of $F$ (or $f$), then one can find an orbit of $F$ (or $f$) such that $p_0=p_n=p$.

\begin{definition} [Limit Point and Limit Set] Let $\{p_n\}_{n=0}^{\infty}$ be an orbit of $f$. If there is a subsequence of $\{p_n\}$ converging to $\overline{p}\in \Sigma$, we say $\overline{p}$ is a limit point of this orbit.

The collection of all limit points of all orbits of $f$ is the limit set of f.
\end{definition}

 We will establish the following results:

\begin{theorem}
f has no generalized periodic points.
\end{theorem}

\begin{theorem}
The limit set of f agrees with the generalized fixed point set of f.
\end{theorem}

\begin{theorem}
Every orbit of f forms a convergent sequence.
\end{theorem}

\begin{theorem}
The limit set of f is contained in a finite union of algebraic curves of degree at most 2.
\end{theorem}

The outline of this work is as follows:\\

In Section 2 we give the notations, terminologies and some elementary but frequently-used lemmas.

In Section 3 we prove Theorem 1.1 and 1.2.

In Section 4 we incorporate the idea of the star-unfolding into a Java program to compute the farthest point set and plot the set where $f$ is not a rational function on a family of centrally symmetric convex octahedra. This construction  works for arbitrary centrally symmetric convex polyhedron, and is necessary to prove Theorem 1.3 and Theorem 1.4 as well.

Finally, in Section 5 we prove both Theorem 1.3 and Theorem 1.4. 

\subsection*{Previous results}
After we finish the paper, we learnt that the proof of many results are already known. Some results in Section 2 are included in \cite{A}: Lemma 2.3 is Theorem (A) on page 72; Case 2 of Lemma 2.5 is Theorem (D) on page 75; Lemma 2.6 is Theorem 2 on page 77. For the reference of any other previously known result, please see the remark after the result. We still keep the proofs that are short and elementary so that the reader may use them to get familiar with the subject.

\subsection*{Acknowledgements}

I would like to thank my adviser Richard Schwartz for many helpful discussions on the problem and feedbacks on the key ideas of proof. The Java program used in this paper is modified from his original program to compute the farthest point map on a regular octahedron.

I am especially grateful to Jo\"el Rouyer, who gives extensive feedbacks on this article, including (and not limited to) comments on the overall structure of this article and suggestions of some better proofs. I would also like to thank him and Costin V\^\i lcu for providing many of the references.

\section{Preliminaries}

On a compact metric space $\Sigma$, denote by $\operatorname{dist}(p,q)$ the distance between two points $p,q\in\Sigma$. 

\subsection{Radius}

\begin{definition}[Radius]

Let $d:\Sigma \to \mathbb{R}$ be a function such that $d(p)=\operatorname{dist}(p,q)$, where $q\in F(p)$. We call $d(p)$ the \textit{radius} at $p$. Notice that $d(p)$ is independent of which $q$ we choose.
\end{definition}

\begin{lemma}
d is nondecreasing on any orbit $\{p_n\}_{n=0}^\infty$ of $F$ through $p$. That is,
$$d(p_0) \leq d(p_1) \leq d(p_2) \leq \dots $$
where $p_n\in F(p_{n-1})$.
\end{lemma}

\begin{proof}
Since $p_{n+1}\in F(p_n)$, $\operatorname{dist}(p_{n-1},p_n)\leq \operatorname{dist}(p_n,p_{n+1})$.
\end{proof}

\begin{lemma}d is continuous on $\Sigma$.
\end{lemma}

\begin{proof} Let $p_1,p_2\in\Sigma$ and $q_1\in F(p_1)$. By definition of $d$ and triangle inequality, $d(p_2)\geq \operatorname{dist}(p_2,q_1)\geq d(p_1)-\operatorname{dist}(p_1,p_2)$. By symmetry, $d(p_1)\geq d(p_2)-\operatorname{dist}(p_1,p_2)$. 

Therefore, $\abs{d(p_1)-d(p_2)}\leq \operatorname{dist}(p_1,p_2)$, so $d$ is continuous.
\end{proof}

\begin{remark}This is a known result in \cite{R1} (see Lemma 1).
\end{remark}

Now suppose $\Sigma$ is a convex polyhedral surface, endowed with the intrinsic path metric. This means $\Sigma$ has a flat metric outside a finite set of \textit{conical points} (better known as vertices), denoted by $\mathscr{C}=\{C_1, C_2, \dots, C_{M}\}$. For each $n\in\{1,2,\dots,M\}$, $C_n$ has a neighborhood isometric to a Euclidean cone of angle $2\pi-\delta_n$, where $\delta_n$ is called the \textit{angular deficit} at $C_n$. Since we assume $\Sigma$ is convex, $0<\delta_n<2\pi$ for all $n$. A consequence of the \textit{Gauss-Bonnet Theorem} is that $\sum_{n=1}^{M} \delta_n= 4\pi$.\\ 

\subsection{Distance Minimizer}

\begin{definition}
Let $p,q\in\Sigma$. A distance minimizer from $p$ to $q$ is a shortest path (i.e. a path with length $\operatorname{dist}(p,q)$) connecting $p$ to $q$. Sometimes we denote such a path by $[p, q]$.
\end{definition} 

\begin{lemma}
Let $[p,q]$ be a distance minimizer. Then $[p,q]$ does not pass through any conical point.
\end{lemma}
\begin{proof}
If $[p, q]$ passes through a conical point $C_n$, then by convexity and triangle inequality, we can construct a shorter path from $p$ to $q$ (see Figure 1, left). This contradicts the definition of the distance minimizer.\end{proof}

\begin{figure}[h]
\centering
\includegraphics[width=0.8\textwidth]{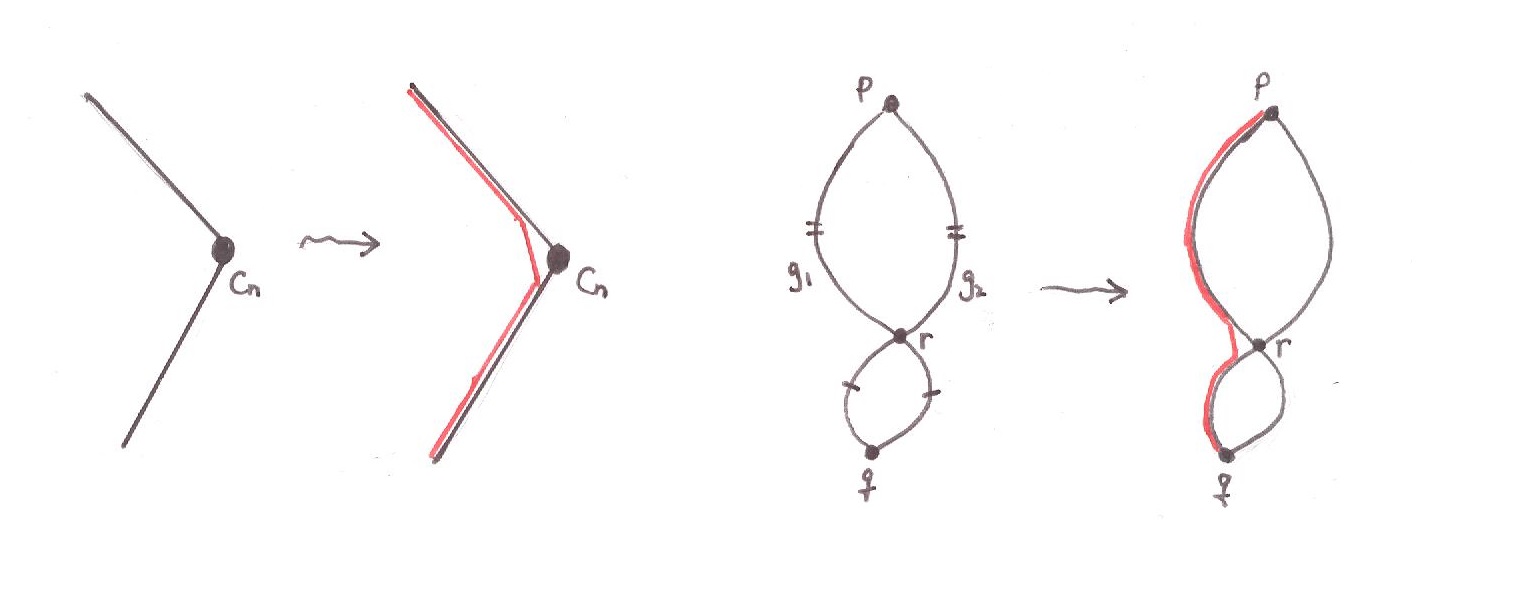}
\caption{The construction of a shorter path than a distance minimizer from $p$ to $q$ in the proof of Lemma 2.3 (left) and Lemma 2.5 (right). }
\end{figure} 

\begin{lemma}
If $q\in F(p)$ and $q\notin \mathscr{C}$, then there are at least three distance minimizers connecting $p$ to $q$.
\end{lemma}
For a proof, see Lemma 3 of \cite{R1}.

\begin{lemma} 
Let $g_1$ and $g_2$ be two distance minimizers emanating from $p$ so that one is not contained in the other. If they have the same endpoint $q$, then $g_1\cap g_2=\{p\}\cup \{q\}$; otherwise, $g_1\cap g_2=\{p\}$.
\end{lemma}

\begin{proof}
We prove in the case $g_1$ and $g_2$ are both from $p$ to $q$, and the same idea works for the other case.
suppose $g_1$ and $g_2$ meet at $r$, where $r\neq p,q$. Then since $g_1$ and $g_2$ are distance minimizers, the two segments from $p$ to $r$, one contained in $g_1$ and the other in $g_2$, must have equal length. Similarly, the two segments from $r$ to $q$ also have equal length. Then by triangle inequality, we can construct a path from $p$ to $q$ that is shorter than $g_1$, as Figure 1 (right) shows. This contradicts that $g_1$ is a distance minimizer.
\end{proof}

\subsection{Lunes}
Let $p \in \Sigma$ and $q\in F(p)$. 
Let $\{g_1,g_2,\dots,g_m\}$ be the collection of all distance minimizers from $p$ to $q$. By Lemma 2.5, $g_1,g_2,\dots,g_m$ divide $\Sigma$ into $m$ connected components, called \textit{lunes}. If $m=1$, then $\Sigma\setminus g_1$ is the only component, whose metric completion is a lune with two edges. Otherwise, the edges of each lune are the two distance minimizers bounding it. By Lemma 2.3, a distance minimizer cannot pass through any conical point, so each conical point must be $p$, $q$ or in the interior of a lune. 

Let $\mathscr{L}_{pq}$ be one of these lunes. Let $\alpha_p$ be the \textit{dihedral angle} of $\mathscr{L}_{pq}$ at $p$, which is the internal angle between its edges. Similarly, let $\alpha_{q}$ be the dihedral angle at $q$. 

\begin{lemma}
\begin{align*}
\alpha_p+\alpha_{q}=\sum_{C_n\in{\mathscr{L}_{pq}}^{\mathrm{o}}}\delta_n \tag{\textasteriskcentered}\label{lenseq}
\end{align*}
where the sum is taken over all conical points in ${\mathscr{L}_{pq}}^{\mathrm{o}}$, the interior of $\mathscr{L}_{pq}$.
\end{lemma}

\begin{proof}
For each $C_n\in {\mathscr{L}_{pq}}^{\mathrm{o}}$, choose a distance minimizer $[p,C_n]$. By Lemma 2.5, $[p,C_n]$ does not intersect the boundary of $\mathscr{L}_{pq}$, so $[p,C_n]\subset\mathscr{L}_{pq}$. Note that ${\mathscr{L}_{pq}}^{\mathrm{o}} - \bigcup\limits_{C_n\in{\mathscr{L}_{pq}}^{\mathrm{o}}}[p, C_n]$ is isometric to the interior of a (2$k$+2)-gon, where $k$ is the number of conical points in ${\mathscr{L}_{pq}}^{\mathrm{o}}$. Figure 2 shows the case when $k=3$. 
\begin{figure}[h]
\centering
\includegraphics[width=0.65\textwidth]{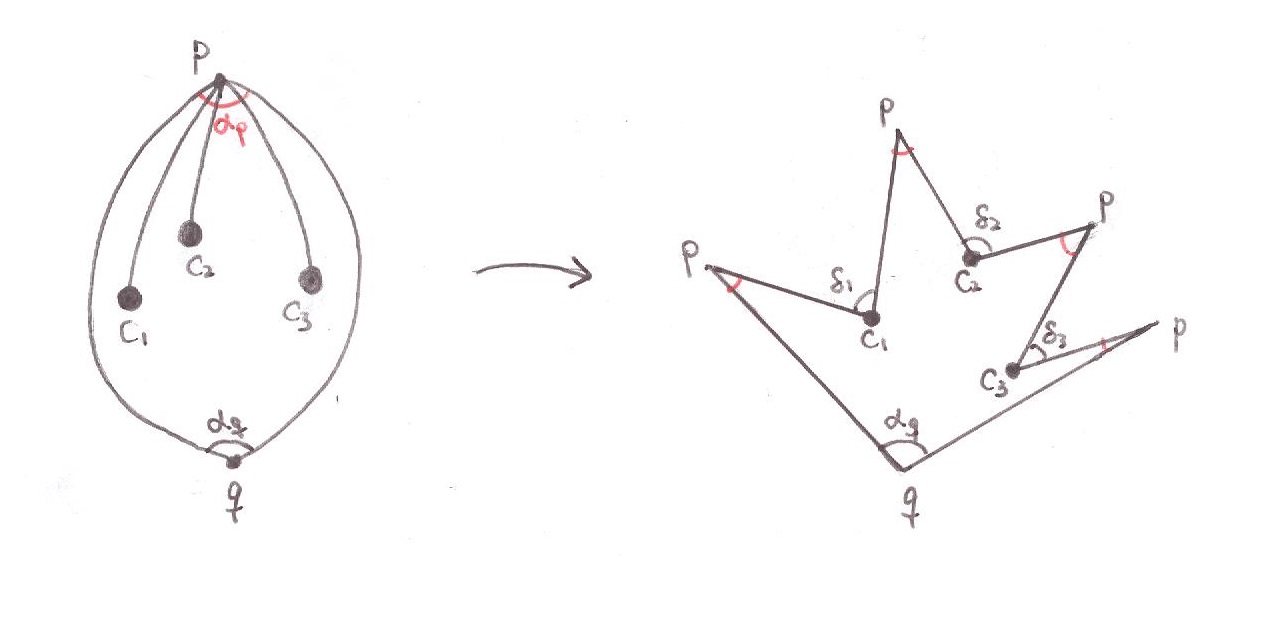}
\caption{${\mathscr{L}_{pq}}^{\mathrm{o}}-\bigcup\limits_{n=1}^3[p, C_n]$ is isometric to the interior of an octagon.}
\end{figure}

Thus, we have
\begin{equation*}
\begin{split}
(2k+2-2)\pi &=\alpha_p+\alpha_{q}+\sum_{C_n\in{\mathscr{L}_{pq}}^{\mathrm{o}}}(2\pi-\delta_n) \\
&=\alpha_p+\alpha_{q}+2k\pi-\sum_{C_n\in{\mathscr{L}_{pq}}^{\mathrm{o}}}\delta_n
\end{split}
\end{equation*}
since both sides are equal to the internal angle sum of a (2$k$+2)-gon. After cancelling $2k\pi$ on both sides, \eqref{lenseq} is established.
\end{proof}

\begin{lemma} Let $\mathscr{L}_{pq}$ be a lune. Then $\alpha_{q}<\pi$.
\end{lemma}

For a proof of Lemma 2.7, see Lemma 3 of \cite{R1}.

\begin{remark} Lemma 2.7 implies the following statement: If $q\notin\mathscr{C}$, then there is at most one point $p$ such that $q\in F(p)$.

To see this, suppose $p_1\neq p_2$ but $q\in F(p_1)$ and $q\in F(p_2)$. Then $p_2$ must lie in the interior of one of the lunes bounded by distance minimizers from $p_1$ to $q$, denoted by $\mathscr{L}_{p_1 q}$. Similarly, $p_1$ must lie in the interior of one of the lunes bounded by distance minimizers from $p_2$ to $q$, denoted by $\mathscr{L}_{p_2 q}$. Let $\alpha_{p_1 q}$ and $\alpha_{p_2 q}$ be the dihedral angles of $\mathscr{L}_{p_1 q}$ and $\mathscr{L}_{p_2 q}$ at $q$. By Lemma 2.7, $\alpha_{p_1 q}+\alpha_{p_2 q}<2\pi$.

On the other hand, by drawing a picture one sees that $\mathscr{L}_{p_1 q}$ and $\mathscr{L}_{p_2 q}$ have nonempty intersection, which implies $\alpha_{p_1 q}+\alpha_{p_2 q}>2\pi$, a contradiction.

This remark is also a case of Theorem 3 in \cite{VZ}.
\end{remark}

\section{Proof of Theorem 1.1 and Theorem 1.2}

Starting from this section, we assume $\Sigma$ is a centrally symmetric convex polyhedral surface, and $\phi$ is the antipodal map defined on $\Sigma$. The number of conical points is even, so we write $M=2N$ for some positive integer $N$.

\subsection{$f$ has no generalized periodic points}

In this section we will show that $f$ has no generalized periodic point.

Notice that if $p\in f^k(p)$, then we also have $p\in F^{2k}(p)$. Thus, any generalized periodic point of $f$ is also a generalized periodic point of $F$. To prove Theorem 1.1, we first show that any generalized periodic point of F has order 2, and then show that such a point must be a generalized fixed point of $f$. 

\begin{lemma} If $p\in F^{k}(p)$ for some $k\geq 2$, then $p\in F^{2}(p)$. 
\end{lemma}

\begin{proof} If $p\in F^k(p)$, then there is a finite sequence $p_1, p_2, \dots, p_{k-1}$ such that $p_1\in F(p)$, $p_n\in F(p_{n-1})$ (where $1<n<k$) and $p\in F(p_{k-1})$. By Lemma 2.1, 
$$d(p)\leq d(p_1) \leq \dots \leq d(p_{k-1}) \leq d(p)$$
Consequently, all $\leq$'s above are equalities. In particular, we have $d(p)\leq d(p_1)$.

By definition of $d$, $\operatorname{dist}(p, p_1)=\operatorname{dist}(p_1, p_2)$. Since $p_2\in F(p_1)$, this implies $p\in F(p_1)$, so $p\in F^2(p)$.

\end{proof}

\begin{lemma}
If $p\in F^{2}(p)$, then $\phi(p)\in F(p)$, or equivalently, $p\in f(p)$.
\end{lemma}

\begin{proof} Let $p$, $q$ be such that $q\in F(p)$ and $p\in F(q)$. We want to show $q=\phi(p)$.

Assume $q\neq\phi(p)$. Let $g$ be a distance minimizer joining $p$ and $\phi(p)$. Since $\Sigma$ is centrally symmetric, $\phi(g)$ is a distance minimizer joining $p$ and $\phi(p)$ distinct from $g$. Note that no distance minimizer from $p$ to $q$ can pass through $\phi(p)$, otherwise it intersects $g$ and $\phi(g)$ at $\phi(p)$, but does not contain both, hence a contradiction to Lemma 2.5. Therefore, $\phi(p)$ is in the interior of some lune $\mathscr{L}_{pq}$. 

Now the loop $g\cup\phi(g)$ divides $\Sigma$ into two parts that are antipodal images of each other. Thus, a conical point and its antipodal image must belong to different parts. The sum of the angular deficits of the conical points inside $g\cup\phi(g)$ is thus  $2\pi-\delta_p$, where $\delta_p$ is the angular deficit at $p$, and equals zero if $p\notin\mathscr{C}$.

\begin{figure}[h]
\centering
\includegraphics[width=0.38\textwidth]{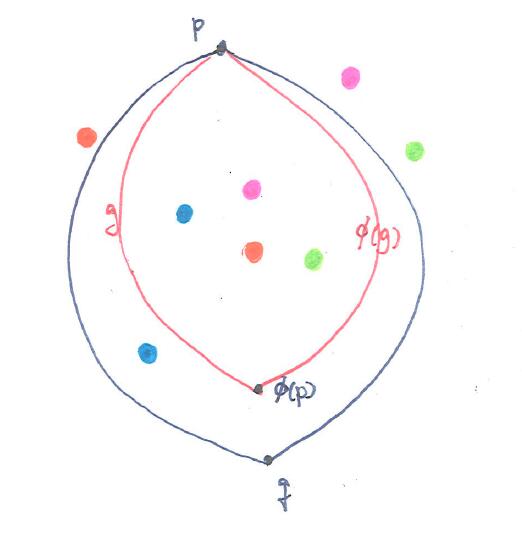}
\caption{This is a sketch of $\Sigma$ for the proof of Lemma 3.2: the colored dots are conical points, where two antipodal conical points have the same color; the two dark blue curves are distance minimizers connecting $p$ to $q$ and they form the boundary of $\mathscr{L}_{pq}$; the red curves $g$ and $\phi(g)$ form a loop $g\cup\phi(g)$ dividing $\Sigma$ into two centrally symmetric parts. Note that no two dots of the same color lie in the same part.}
\end{figure}

By lemma 2.5, $g$ and $\phi(g)$ only intersect the boundary of $\mathscr{L}_{pq}$ at $p$, so the conical points inside $g\cup\phi(g)$ are also inside $\mathscr{L}_{pq}$. Therefore, the sum of the angular deficits of the conical points inside $\mathscr{L}_{pq}$ and that of $\phi(p)$ is no less than $2\pi-\delta_p+\delta_p=2\pi$ (see Figure 3 for a demonstration). According to Lemma 2.6, $\alpha_p+\alpha_{q}\geq 2\pi$.

On the other hand, since $q\in F(p)$ and $p\in F(q)$, it follows from lemma 2.7 that $\alpha_p<\pi$ and $\alpha_{q}<\pi$, and consequently $\alpha_p+\alpha_q< 2\pi$, a contradiction. This proves the statement.
\end{proof}

Theorem 1.1 then follows from Lemma 3.1 and  Lemma 3.2.

\subsection{Limit Set of $f$ is the Generalized Fixed Point Set  of $f$}
In this section, we will show that the limit set of $f$ is equivalent to the set of generalized fixed points of $f$. Clearly, if $p$ is a generalized fixed point of $f$, it is also a limit point of $f$ (of the orbit $p, p, p, \dots$). Hence, it suffices to show that any limit point of $f$ must be a generalized fixed point of $f$.

Recall $d$ is continuous (Lemma 2.2) on a compact set, so it has an upper bound. Let $\{p_n\}_{n\in\mathbb{N}}$ be an orbit of $f$. The sequence $\{d(p_n)\}_{n\in\mathbb{N}}$ is monotone (Lemma 2.1) and bounded, hence it converges to some finite number. Then we get the following lemma:

\begin{lemma} Suppose $\overline{p}$ is a limit point of an orbit $\{p_n\}_{n\in\mathbb{N}}$ of $f$, and $\{d(p_n)\}_{n\in\mathbb{N}}$ converges to $L$. Then $d(\overline{p})=L$.
\end{lemma}

Now we give a proof of Theorem 1.2.

\begin{proof} Let $\overline{p}$ be a limit point of an orbit $\{p_{n}\}$ of $f$. Choose a subsequence $\{p_{n_k}\}\subset\{p_{n}\}$ converging to $\overline{p}$. Without loss of generality, we may assume the sequence $\{\phi(p_{n_k+1})\}$ converges to some point $\overline{q}$, otherwise we replace it (and $\{p_{n_k}\}$ accordingly) by a convergent subsequence. Note that $\phi(\overline{q})$ is also a limit point of the orbit $\{p_n\}$,
so 
  
$$d(\overline{p})=d(\phi(\overline{q}))=d(\overline{q})$$ by Lemma 3.3 and symmetry.

Therefore,
$$\operatorname{dist}(\overline{p},\overline{q})=\lim_{n_{k}\to \infty}\operatorname{dist}(p_{n_{k}}, \phi(p_{n_{k}+1}))=\lim_{n_k\to \infty}d(p_{n_{k}})=d(\overline{p})=d(\overline{q})$$
where the second equality follows from $\phi(p_{n_k+1})\in F(p_{n_k})$ and the third equality from continuity of $d$ (Lemma 2.2).

This implies $\overline{q}\in F(\overline{p})$ and $\overline{p}\in F(\overline{q})$. By Lemma 3.2, $\overline{p}\in f(\overline{p})$, hence $\overline{p}$ is a generalized fixed point of $f$.
\end{proof}

\section{Star Unfolding and a Coordinate Representation of $f$}

\subsection{Computation of the Farthest Point Set}

The goal of this section is to introduce an idea based on which we wrote a computer program in Java to compute the farthest point set. We achieve this through a map that ``unfolds $\Sigma$ onto the plane", known as the star-unfolding. For simplicity, we will demonstrate the idea of computation for non-conical points, and the remaining finite cases can be worked out with very few adjustments. 

Let $M$ be a $(G,X)$ manifold. Choose a point $p\in M$ and a coordinate chart $\psi_p:U_p\to X$ containing $p$ (think of $\Sigma\setminus\mathscr{C}$ as a flat manifold). Let $q\in M$ be any other point, and $\gamma_{p,q}\subset M$ be a path from $p$ to $q$. Then $\psi_p$ can be analytically continued along $\gamma_{p,q}$ so that its domain is extended to cover $q$. The value of $\psi_p(q)$ depends only on the homotopy class of the path $\gamma_{p,q}$. Thus, if $M$ is simply connected, the map $q\mapsto\psi_p(q)$ is well-defined, called the \textit{developing map}.  In general, the developing map is well-defined on the universal cover $\widetilde{M}$, interpreted as the space of homotopy classes of paths in $M$ emanating from $p$. Developing map is uniquely determined by the basepoint $p$ and the chart $\psi_p$, but changing $p$ or $\psi_p$ only composes the original developing map with an element of $G$. The reader may refer to Chapter 3.5 of \cite{Th} for more details.

\begin{definition}[Angle from one geodesic to another] Let $g_1$ and $g_2$ be two geodesics meeting at a point $p\in\Sigma$. Suppose we rotate $g_1$ about $p$ on $\Sigma$ counter-clockwisely by an angle $\theta> 0$ so that it overlaps with $g_2$ near $p$. We call the smallest such $\theta$ the angle from $g_1$ to $g_2$ at $p$.
\end{definition}

Let $p\in\Sigma\setminus\mathscr{C}$ be a non-conical point. For each $n$ ($1\leq n\leq 2N$), choose a distance minimizer $[\phi(p), C_n]$ from the antipodal point $\phi(p)$ to $C_n$. Rename the conical points other than $C_1$ if necessary, we may assume that the angle from $[\phi(p), C_1]$ to $[\phi(p), C_n]$ at $\phi(p)$ is increasing with respect to $n$. 

By Lemma 2.5, if $n\neq m$, then $[\phi(p), C_n]\cap [\phi(p), C_m]=\{\phi(p)\}$. Thus, $\tau_p:=\bigcup\limits_{n=1}^{2N}[\phi(p), C_n]$ is an embedded tree in $\Sigma$, and $\Sigma\setminus\tau_p$ is simply connected. Then there exists a well-defined developing map $Dev_p:\Sigma\setminus\tau_p\to\mathbb{E}$ (think of $\Sigma\setminus\tau_p$ as a flat manifold). In the sequel, we refer to $Dev_p$ as a \textit{star-unfolding}, since its image resembles the shape of a star. In \cite{AO}, it is shown that a star-unfolding of any convex polyhedron is an embedding, so $Dev_p$ is an embedding. Figure 4 is a demonstration when $\Sigma$ is an octahedron, and $p$ is marked red in the first picture.

Let $\overline{Dev_p}$ be the multivalued-extension of $Dev_p$ to $\Sigma$ by continuity. Then $\overline{Dev_p}(\Sigma)$ is a Euclidean $4N$-gon $\bigstar_p$. Note that $\overline{Dev_p}(\Sigma)$ is multi-valued on $[\phi(p),C_n]-\{C_n\}$. In particular, at $\phi(p)$, $\overline{Dev_p}(\Sigma)$ has $2N$ images. We label them by $\phi_1(p), \phi_2(p), \dots, \phi_{2N}(p)$, such that $\phi_n(p)$ is adjacent to $\overline{Dev_p}(C_n)$ and $\overline{Dev_p}(C_{n+1})$ for all $n$, where the index $n+1$ should be understood modulo $2N$.

Note that the construction above also works when $p$ (and hence $\phi(p)$) is a conical point, by replacing all the $2N$'s with $(2N-1)$'s.

Recall that if $q\in f(p)$, then $q\in\mathscr{C}$ or there are at least three distance minimizers joining $\phi(p)$ and $q$ (Lemma 2.4). In the latter case, $q$ is not in the relative interior of $\tau_p$, otherwise for some $n$, $[\phi(p), C_n]$ intersects at $q$ a distance minimizer from $\phi(p)$ to $q$ but does not contain it, hence contradicts Lemma 2.5. Let $\mathcal{S}:=\{q\in\Sigma-\tau_p:$\textit{ there are at least three distance minimizers from $\phi(p)$ to $q\}$}. By the reasoning above, $f(p)\subset \mathcal{S}\cup\mathscr{C}$. We now turn to computing the set $\mathcal{S}$.  

\begin{figure}[h]
\centering
\includegraphics[width=0.7\textwidth]{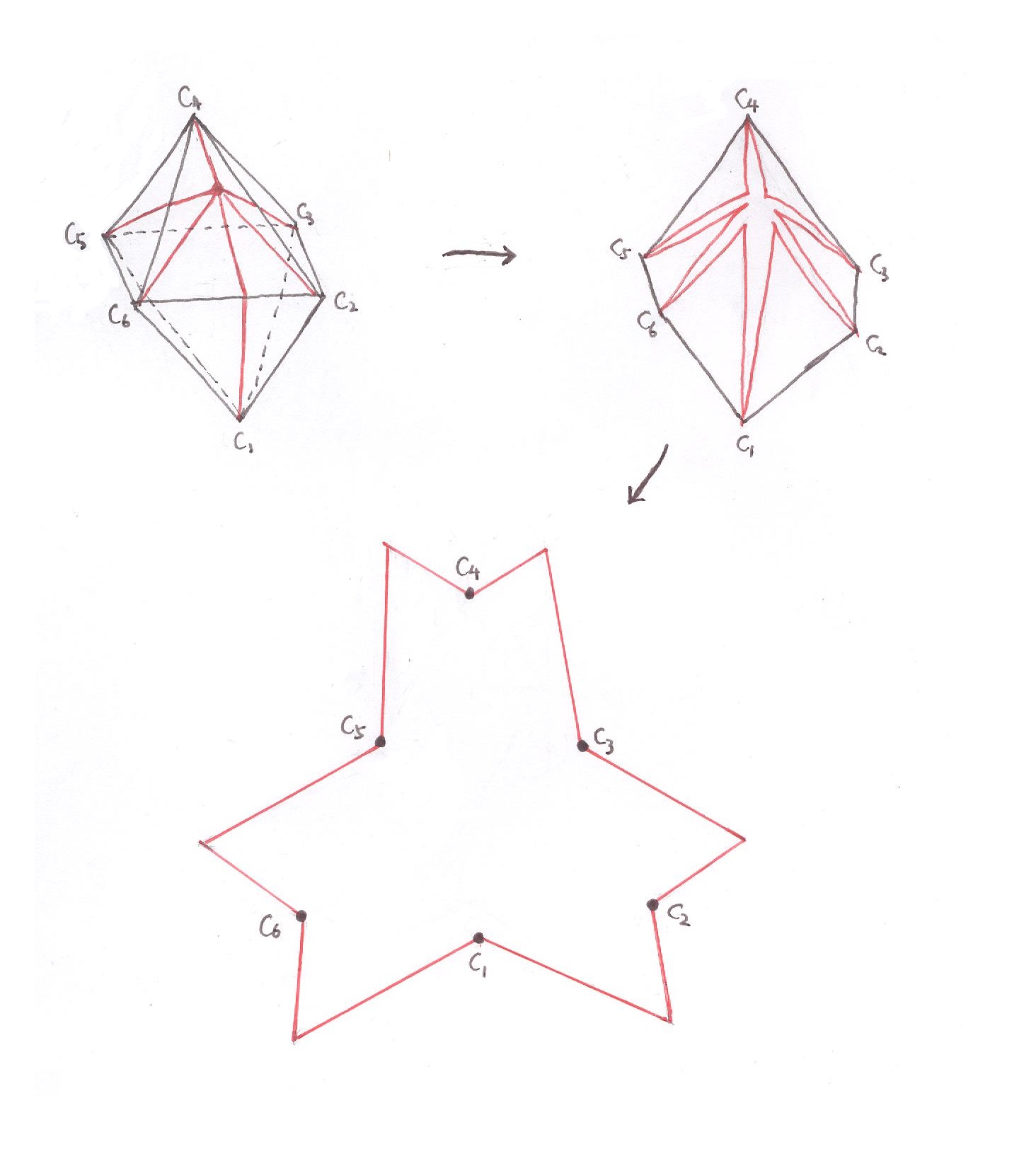}
\caption{The first picture is a sketch of $\tau_p$ in red. The conical points are named based on the rule in paragraph 2. The second picture is a visualization of ``cutting $\Sigma$ open along $\tau_p$", and when we flatten this onto the plane, we get the 12-gon in the third picture, the image of $Dev_p$. }
\end{figure}

\begin{definition}[$\bigstar_p$-path] Let $X, Y\in\bigstar_p$. A path joining $X$ and $Y$ is a $\bigstar_p$-path if it is contained in the interior of $\bigstar_p$ except for its endpoints.
\end{definition}

\begin{definition}[Good Triple] Let $i, j, k \in \{1,2,\dots,2N\}$ be distinct. We say that an unordered triple $\{i,j,k\}$ is a good triple at $p$ if 

1. $q_{ijk}$ is in the interior of $\bigstar_p$, where $q_{ijk}$ is the circumcenter of $\phi_i(p)$, $\phi_j(p)$ and $\phi_k(p)$.

2. $[q_{ijk},\phi_i(p)]$, $[q_{ijk},\phi_j(p)]$ and $[q_{ijk},\phi_k(p)]$ are $\bigstar_p$-paths, where $[q_{ijk},\phi_i(p)]$ is the Euclidean segment joining $q_{ijk}$ and $\phi_i(p)$, and similar for $[q_{ijk},\phi_j(p)]$ and $[q_{ijk},\phi_k(p)]$.

3. If $n\neq i, j, k$ and $[q_{ijk},\phi_n(p)]$ is a $\bigstar_p$-path, then $\norm{\phi_i(p)-q_{ijk}}\leq\norm{\phi_n(p)-q_{ijk}}$, where \norm{.} is the Euclidean norm, $q_{ijk}$, $\phi_i(p)$ and $\phi_i(p)$ are viewed as vectors in $\mathbb{R}^2$.
\end{definition}

Let $q\in\mathcal{S}$. We show that $q$ comes from a good triple. Choose three distinct distance minimizers $g_1,g_2,g_3$ from $\phi(p)$ to $q$. By Lemma 2.5, $g_1,g_2$ and $g_3$ intersect $\tau_p$ only at $\phi(p)$. Thus, $\overline{Dev_p}(g_1)$, $\overline{Dev_p}(g_2)$ and $\overline{Dev_p}(g_3)$ are $\bigstar_p$-paths joining $Dev_p(q)$ to $\phi_i(p),\phi_j(p)$ and $\phi_k(p)$ respectively for some triple $\{i,j,k\}$. Since $g_1,g_2$ and $g_3$ have the same length, this implies $Dev_p(q)$ is equidistant from $\phi_i(p),\phi_j(p)$ and $\phi_k(p)$, that is, $Dev_p(q)=q_{ijk}$. Since $q\in\Sigma-\tau_p$, $q_{ijk}$ is in the interior of $\bigstar_p$. Finally, if $[Dev_p(q),\phi_n(p)]$ is a $\bigstar_p$-path for some $n\neq i,j,k$, then its preimage under $\overline{Dev_p}$ is a path joining $\phi(p)$ and $q$, hence is no shorter than $g_1$ because $g_1$ is a distance minimizer. Therefore, $\norm{\phi_i(p)-q_{ijk}}\leq\norm{\phi_n(p)-q_{ijk}}$. As all the conditions in Definition 4.3 are satisfied, $\{i,j,k\}$ is a good triple at $p$.

Since $Dev_p$ is an embedding, every good triple $\{i,j,k\}$ determines a unique point ${Dev_p}^{-1}(q_{ijk})\in\mathcal{S}$. Combined with last paragraph, we see that the cardinality of $\mathcal{S}$ is no greater than the number of good triples at $p$ (in fact, this number is $\leq 2N-2$ by Lemma 7, \cite{R2}), hence $\mathcal{S}$ is a finite set.

Let $M_1=\max\limits_{q\in\mathcal{S}}\operatorname{dist}(\phi(p),q)$, and $M_2=\max\limits_{1\leq n\leq 2N}\operatorname{dist}(\phi(p),C_n)$. Then $f(p)$ is given by one of the following:

$f(p)=\{{\overline{Dev_p}}^{-1}(q_{ijk}): \{i,j,k\}\text{ is a good triple, } \norm{\phi_i(p)-q_{ijk}}=M_1\}$ if $M_1>M_2$;

$f(p)=\{C_n: \operatorname{dist}(\phi(p),C_n)=M_2\}$ if $M_1<M_2$; 

$f(p)$ is the union of the two sets above if $M_1=M_2$.

\subsection{$f$ is Piecewise Rational}

In this section, we present a method to cut $\Sigma$ into finitely many pieces, and in the interior of each piece $f$ is some rational function. The rationality result dates back to Jo\"el Rouyer's work in \cite{R1}, but we use a different construction to get a better description of the regions on which $f$ is rational. This construction is also necessary to prove Theorem 1.3 and 1.4. 

Fix a basepoint $p\in \Sigma\setminus\mathscr{C}$ and a coordinate chart $\psi_p:U_p\to\mathbb{E}$ of $p$. For any $q\in U_p$, choose $2N$ distance minimizers $[\phi(q), C_1]$, $[\phi(q), C_2],\dots, [\phi(q), C_{2N}]$, where the indices of the conical points other than $C_1$ are assigned according to the same rule as in Section 4.1: the angle from $[\phi(q), C_1]$ to $[\phi(q), C_n]$ at $\phi(p)$ is increasing as $n$ increases from $2$ to $2N$. For each $q$, let $Dev_q: \Sigma\setminus\bigcup\limits_{n=1}^{2N}[\phi(q), C_i]\to\mathbb{E}$ be the unique star unfolding such that $Dev_q=\psi_p$ in a small neighborhood of $q$. We also denote the images of $\phi(q)$ under $\overline{Dev_q}$ by $\phi_1(q),\phi_2(q),\dots,\phi_{2N}(q)$, so that $\phi_n(q)$ is adjacent to $\overline{Dev_q}(C_n)$ and $\overline{Dev_q}(C_{n+1})$. 

Consider the map $\psi_p(q)\mapsto\phi_n(q)$ for each $n$. Intuitively, if we perturb $q$, $\psi_p(q)$ and $\phi_n(q)$ should move by the same distance, except when there are multiple ways to assign indices to conical points (for instance, if there are at least two distance minimizers from $\phi(q)$ to some conical point, then there are at least two ways to assign indices), then there is an ambiguity in which way we should choose. We need a more precise statement to take care of this situation:

\begin{lemma}

There is a subdivision of $\Sigma$ into finitely many closed regions $R_s$ such that: $(1)$ there is an isometric embedding $\Psi_s: {R_s}^{\mathrm{o}}\to\mathbb{E}$, where ${R_s}^{\mathrm{o}}$ is the interior of $R_s$; $(2)$ for each $q\in {R_s}^{\mathrm{o}}$, if we choose the unique star unfolding $Dev_q$ such that $Dev_q=\Psi_s$ in a neighborhood of $q$, then for any $n\in\{1,2,\dots,2N\}$, the map $\Psi_s(q)\mapsto\phi_n(q)$ is an isometry on $\Psi_s({R_s}^{\mathrm{o}})$.

\end{lemma}

\begin{proof}

If $C_n$ is a conical point, then the cut locus of $C_n$, the closure of the set of points in $\Sigma$ with more than one distance minimizers to $C_n$, is a finite tree $\mathcal{T}_n$ whose leaves are $\mathscr{C}\setminus C_n$. Let $\mathscr{T}=\bigcup\limits_{n=1}^{2N}\mathcal{T}_n$. Then $\Sigma\setminus\mathscr{T}$ is a finite union of disjoint open regions ${R_s}^{\mathrm{o}}$ . In Figure 5 we sketch $\mathscr{T}$, when $\Sigma$ are the surfaces of three anti-prisms with regular triangular bases and different heights.

\begin{figure}[h]
\centering
\includegraphics[width=1.0\textwidth]{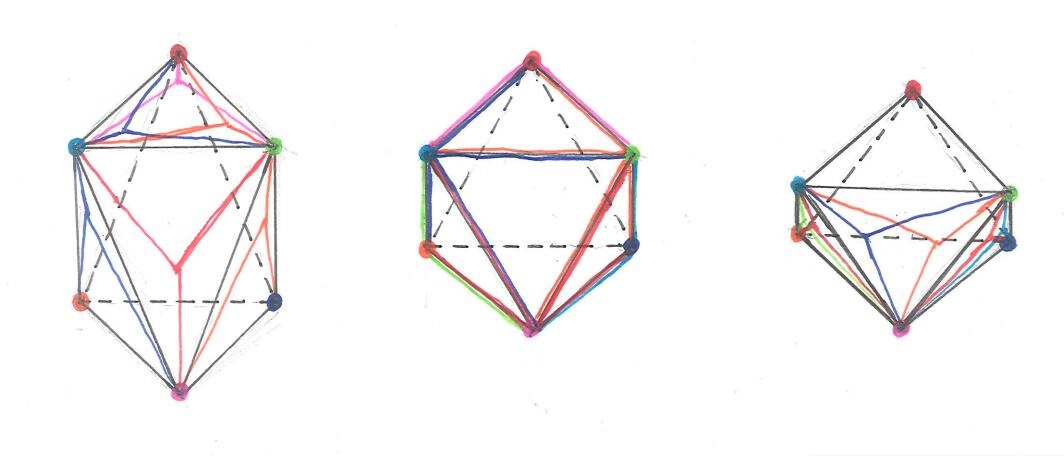}
\caption{A sketch of $\mathscr{T}$ when $\Sigma$ are anti-prisms with regular triangular bases of decreasing heights from left to right. The second anti-prism is a regular octahedron. The colored dots are conical points, and their corresponding cut loci trees are sketched with the same color. Observe that in case of a regular octahedron, $R_s$ is a regular triangle for all $s$. In the other two cases, $R_s$ can be a triangle, a quadrilateral or a hexagon. For the cleanness of the picture, we only sketch $\mathscr{T}$ on four front faces of $\Sigma$. The reader may get the rest half of the picture by symmetry.}
\end{figure}

Now we show ${R_s}^{\mathrm{o}}$ is simply connected (this would imply ${R_s}^{\mathrm{o}}$ is a topological disk by \textit{Uniformization Theorem}): Without loss of generality, suppose there is a nontrivial simple loop $l$ contained in ${R_s}^{\mathrm{o}}$, then since $\Sigma$ is a topological sphere, we can find two points lying on different sides of $l$ (\textit{Jordan Curve Theorem}), such that they belong to the boundaries of two regions distinct from $R_s$. Since $\mathscr{T}$ is connected, they are joined by a path in $\mathscr{T}$, but this path cuts $l$, a contradiction to $l\subset{R_s}^{\mathrm{o}}$. 

Since ${R_s}^{\mathrm{o}}$ is simply connected, let $\Psi_s$ be a developing map from ${R_s}^{\mathrm{o}}$ to $\mathbb{E}$. Then $R_s$ and $\Psi_s$ satisfy $(1)$ in the claim. 

We claim that throughout ${R_s}^{\mathrm{o}}$, there is a consistent way of assigning indices to all conical points. Suppose not, let $q_1,q_2\in {R_s}^{\mathrm{o}}$ be two points such that a certain conical point is assigned different indices at $q_1$ and $q_2$. Choose a path from $q_1$ to $q_2$ contained in ${R_s}^{\mathrm{o}}$. Then there exists a point $q_3$ on this path such that there are at least two distance minimizers from $\phi(q_3)$ to this conical point, so $q_3\in\mathscr{T}$, a contradiction.

Now for each $q\in{R_s}^{\mathrm{o}}$, We choose the unique star unfolding $Dev_q$ satisfying $Dev_q=\Psi_s$ in a neighborhood of $q$. Let $\phi_n(q)$ be the image of $\phi(q)$ under $\overline{Dev_q}$ adjacent to $\overline{Dev_q}(C_n)$ and $\overline{Dev_q}(C_{n+1})$.

Choose a basepoint $q_0\in R_s^{\mathrm{o}}$. Let $\gamma_{0}$ be a path joining $q_0$ and $\phi(q_0)$ such that $\overline{Dev_{q_0}}(\gamma_0)$ is a $\bigstar_{q_0}$-path joining $Dev_{q_0}(q_0)=\Psi_s(q_0)$ and $\phi_n(q_0)$. By symmetry, $R_{s'}:=\phi(R_s)$ is also a convex region with an isometric embedding $\Psi_{s'}:R_{s'}^{\mathrm{o}}\to\mathbb{E}$. Let $\Psi_{{s'},\gamma_{0}}:R_{s'}^{\mathrm{o}}\to\mathbb{E}$ be the analytic continuation of $\Psi_{s'}$ along $\gamma_{0}$. Let $I_n=\Psi_{s',\gamma_n}\circ\phi\circ{\Psi_s}^{-1}$. $I_n$ can be viewed as a ``coordinate representation of $\phi$ induced by $\gamma_{0}$''. It remains to show that $I_n$ coincides with the map $\Psi_s(q)\mapsto\phi_n(q)$ on $\Psi_s({R_s}^{\mathrm{o}})$ 

For each $q\in{R_s}^{\mathrm{o}}$, let $\gamma_q$ be a path from $q$ to $\phi(q)$ obtained by concatenating the following three: first a path from $q$ to $q_0$ contained in ${R_s}^{\mathrm{o}}$, then $\gamma_{0}$, and finally a path from $\phi(q_0)$ to $\phi(q)$ contained in ${R_{s'}}^{\mathrm{o}}$.

Observe that $I_n(\Psi_s(q))=\phi_n(q)$ if and only if $\gamma_q$ is homotopic to the preimage of a $\bigstar_q$-path joining $\Psi_s(q)$ and $\phi_n(q)$ under $\overline{Dev_q}$. By construction, $q_0$ satisfies this condition. 

Suppose $I_n(\Psi_s(q))\neq\phi_n(q)$ for some $q\in{R_s}^{\mathrm{o}}$. Then on any path joining $q_0$ and $q$ contained in ${R_s}^{\mathrm{o}}$, we can find some $q'$ with at least two non-homotopic classes of $\gamma_{q'}$ satisfying the condition above. This means at $q'$, the conical points can be indexed in two different ways, which is impossible. This finishes the proof of the lemma.
\end{proof}

\begin{remark}By a similar argument, we can show that for any $n$, the map $\Psi_s(q)\mapsto\overline{Dev_q}(C_n)$ is constant on ${R_s}^{\mathrm{o}}$: Let $\gamma_{q,n}$ be any path obtained by concatenating a path from $q$ to $q_0$ contained in ${R_s}^{\mathrm{o}}$ with the distance minimizer $[q_0,C_n]$. Then modify the argument in the proof of Lemma 4.1 by replacing all $\gamma_q$ with $\gamma_{q,n}$.

\end{remark}

\begin{definition} Define $C_n(s):=\overline{Dev_q}(C_n)$, where $q\in {R_s}^{\mathrm{o}}$. By the remark above, $C_n(s)$ is independent of which $q$ we choose.
\end{definition}

Notice that there is a unique way to extend $I_n$ to an orientation reversing Euclidean isometry. In the following text, we use $I_n$ to refer to this orientation reversing Euclidean isometry instead.\\

The following proposition is not relevant to the main results, but we present it for any interested reader.

\begin{prop}  ${R_s}^{\mathrm{o}}$ is a convex polygon for any $s$. 
\end{prop}

\begin{proof} In the proof of Lemma 4.1, we have shown that ${R_s}^{\mathrm{o}}$ is a topological disk. As ${R_s}^{\mathrm{o}}$ is bounded by straight segments, it is a polygon. Let $v\in\mathscr{T}$ be a vertex of ${R_s}^{\mathrm{o}}$. To show ${R_s}^{\mathrm{o}}$ is convex, we show that the interior angle of $R_s$ at $v$ is strictly less than $\pi$. 

By Theorem 10.2 in \cite{AO}, for any $n$, the image of $\mathcal{T}_n$ under the star unfolding $\overline{Dev_{\phi(c_n)}}$ is the Voronoi diagram of the $2(N-1)$ images of $C_n$ under $\overline{Dev_{\phi(C_n)}}$. Since Voronoi regions are convex, it suffices to consider the case where $v$ is a conical point with cone angle $\theta$, where $\theta<2\pi$.

Consider all conical points that are joined to $v$ by a unique distance minimizer. We index them by $C_1,C_2,\dots,C_K$ such that the angle from $[C_1,v]$ to $[C_n,v]$ at $v$ is positive and increasing when $n$ increases from $2$ to $K$, where $[C_n,v]$ ($1\leq n\leq K$) is the unique distance minimizer from $C_n$ to $v$. Note that every edge of the polyhedron $\Sigma$ incident to $v$ is such a distance minimizer. Then in a neighborhood of $v$, $[C_n,v]$ and $[C_{n+1},v]$ lie on the same face of the polyhedron $\Sigma$, otherwise there is an edge $[C_m, v]$ between them, a contradiction to the way we index them. Thus, the angle from $[C_n,v]$ to $[C_{n+1},v]$ at $v$ is at most $\frac{\theta}{2}$.

Let $r_n$ be the ray emanating from $v$, making an angle of $\frac{\theta}{2}$ with $[C_n,v]$. Now we show that there is a neighborhood $N_v$ of $v$ such that $N_v\cap\bigcup\limits_{n=1}^K r_n$ is contained in $N_v\cap\mathscr{T}$: if $p$ is a point on $r_n$ sufficiently close to $v$, then there are two geodesics from $C_n$ to $p$ of equal length, since the two triangles with vertices $C_n$, $v$ and $p$ are $SAS$-congruent. If we move $p$ away from $v$ along $r_n$, these two geodesics remain distance minimizers until the moment when there are three distance minimizers from $C_n$ to $p$.

Therefore, the interior angle of $R_s$ at $v$ is no larger than the angle from $r_n$ to $r_{n+1}$ (the indices should be understood mod $K$) at $v$, which is also the angle from $[C_n,v]$ to $[C_{n+1},v]$ at $v$. This angle is no larger than $\frac{\theta}{2}$ and hence strictly less than $\pi$. So ${R_s}^{\mathrm{o}}$ is a convex polygon.
\end{proof}

\begin{lemma} Suppose $i,j\in\{1,2,\dots,2N\}$ are distinct. Then $I_{j}{I_{i}}^{-1}$ is either a translation or a rotation by $\sum_{i+1}^j\delta_n$, the indices being understood mod $2N$.
\end{lemma}
\begin{proof} Choose any $q\in {R_s}^{\mathrm{o}}$. By construction, $I_i$ and $I_j$ are induced by two paths $\gamma_i$ and $\gamma_j$, such that $\overline{Dev_q}(\gamma_i)$ and $\overline{Dev_p}(\gamma_j)$ are $\bigstar_q$-paths from $\Psi_s(q)$ to $\phi_i(q)$ and $\phi_j(q)$, respectively. Then $\sum_{i+1}^j\delta_n$ is the sum of the angular deficits of the conical points enclosed by the loop ${\gamma_{i}}^{-1}\gamma_{j}$. It follows that $I_{j}{I_{i}}^{-1}$ is a translation if and only if this sum is exactly $2\pi$.
\end{proof}

Now we associate to every point in ${R_s}^{\mathrm{o}}$ a pair of coordinates in $\mathbb{R}^2$ using $\Psi_s$. 

Let $i,j,k\in\{1,2,\dots,2N\}$ be mutually distinct. Consider a map $f_{ijk}$ that sends $(x,y)$ to the point equidistant from $I_i(x,y)$, $I_j(x,y)$ and $I_k(x,y)$ if they are non-collinear. 

\begin{lemma} 

$$f_{ijk}(x,y)=(\frac{Q_1(x,y)}{Q_3(x,y)}, \frac{Q_2(x,y)}{Q_3(x,y)})$$
where $Q_1, Q_2$ and $Q_3$ are polynomials in $x$ and $y$ of degree at most 2 depending on $i,j$ and $k$.

\end{lemma}
\begin{proof}
According to Lemma 4.2, $I_{j}{I_{i}}^{-1}$, $I_{i}{I_{k}}^{-1}$ and $I_{k}{I_{j}}^{-1}$ are either rotation or translation, depending on the sum of the angular deficits of the conical points enclosed by ${\gamma_{i}}^{-1}\gamma_{j}$, ${\gamma_{k}}^{-1}\gamma_{i}$ and ${\gamma_{j}}^{-1}\gamma_{k}$. Since the total sum of the angular deficits is $4\pi$, at most one of these three isometries can be translation. Thus, without loss of generality, we assume $I_{j}{I_{i}}^{-1}$ is the rotation about $z_{ij}$ by $\theta_{ij}$, and $I_{k}{I_{i}}^{-1}$ is the rotation (with same orientation) about $z_{ik}$ by $\theta_{ik}$. 

Let $l_{ji}$ be the equidistant line from $I_{i}(x, y)$ and $I_{j}(x, y)$. Then $l_{ij}$ passes through $z_{ij}$ (whose coordinates $(c_1,c_2)$ are independent of $x$ or $y$), and the image of $I_{i}(x, y)$ under the rotation about $z_{ij}$ by $\frac{\theta_{ij}}{2}$, whose coordinates $(X_j,Y_j)$ depend linearly in $x$ and $y$. Similarly, the equidistant line $l_{ik}$ from $I_{i}(x, y)$ and $I_{k}(x, y)$ passes through $z_{ik}=(c_3,c_4)$ and the image of $I_{i}(x, y)$ under the rotation about $z_{ik}$ by $\frac{\theta_{ik}}{2}$, with coordinates $(X_k,Y_k)$. Then the circumcenter of $I_{i}(x, y)$, $I_{j}(x, y)$ and $I_{k}(x, y)$, provided they are non-collinear, is the intersection of $l_{ij}$ and $l_{ik}$. Its coordinates are given by $(\frac{v_x}{v_z},\frac{v_y}{v_z})$, where $v_x, v_y$ and $v_z$ are the $x,y,z$-components of the vector $$((X_j,Y_j,1)\times(c_1,c_2,1))\times((X_k,Y_k,1)\times(c_3,c_4,1))$$
Then it is not hard to check that $v_x, v_y$ and $v_z$ are at most quadratic polynomials in $x$ and $y$.
\end{proof}

For every $s$ and mutually distinct $i,j,k$, we define $$D(s)_{ijk}(p)=\norm{\phi_i(p)-f_{ijk}\circ\Psi_s(p)}=\norm{\phi_j(p)-f_{ijk}\circ\Psi_s(p)}=\norm{\phi_k(p)-f_{ijk}\circ\Psi_s(p)}$$ whenever $f_{ijk}$ is defined at $\Psi_s(p)$. 

Now consider the set of points $p\in {R_s}^{\mathrm{o}}$ satisfying an equation of any of the following three types:

\textbf{\textit{Type 1.}} $D(s)_{ijk}(p)=D(s)_{abc}(p)$, where $\{i,j,k\}$ and $\{a,b,c\}$ are distinct triples;

\textbf{\textit{Type 2.}} $D(s)_{ijk}(p)=\norm{\phi_n(p)-C_n(s)}$ for $C_n(s)$ in Definition 4.3;

\textbf{\textit{Type 3.}} $\norm{\phi_m(p)-C_m(s)}=\norm{\phi_n(p)-C_n(s)}$ where $m\neq n$.\\

Let $\mathcal{Z}_s\subset {R_s}^{\mathrm{o}}$ be the set of ``valid solutions'' to these equations as follows:

If $p$ is a solution to a \textit{Type 1} equation, then $p$ is valid if both $\{i,j,k\}$ and $\{a,b,c\}$ are good triples, and $D(s)_{ijk}(p)=d(p)$; If $p$ is a solution to a \textit{Type 2} equation, then $p$ is valid if $\{i,j,k\}$ is a good triple and $\norm{\phi_n(p)-C_n(s)}=d(p)$; If $p$ is a solution to a \textit{Type 3} equation, then $p$ is valid if $\norm{\phi_n(p)-C_n(s)}=d(p)$.

\begin{lemma} Let $\mathcal{Z}=\bigcup\limits_{s}\mathcal{Z}_s$. Then $f$ is single-valued and continuous on $\Sigma\setminus(\mathscr{T}\cup\mathcal{Z})$. 
\end{lemma}

\begin{proof}
Assume $f$ is multi-valued at $p\in\Sigma\setminus\mathscr{T}$. Choose $s$ so that $p\in {R_s}^{\mathrm{o}}$. Let $q_1, q_2\in f(p)$ be distinct. Consider the following three cases:

1). $q_1,q_2\notin\mathscr{C}$. Then by Lemma 2.4, one can find three distinct distance minimizers from $\phi(p)$ to $q_1$, and another three distinct distance minimizers from $\phi(p)$ to $q_2$. Furthermore, $\operatorname{dist}(\phi(p),q_1)=\operatorname{dist}(\phi(p),q_2)$. Thus, there are two good triples $\{i,j,k\}$ and $\{a,b,c\}$ at $p$, such that $D(s)_{ijk}(p)=D(s)_{abc}(p)$. This implies $p$ is a valid solution to an equation of {\textit{Type 1}}.

2). Either $q_1\in\mathscr{C}$ or $q_2\in\mathscr{C}$, but not both. Then it is not hard to see that $p$ is a valid solution to an equation of \textit{Type 2}.

3). $q_1,q_2\in\mathscr{C}$. In this case, $p$ is a valid solution to an equation of \textit{Type 3}.

In short, the set of points in ${R_s}^{\mathrm{o}}$ where $f$ is multivalued is contained in $\mathcal{Z}_s$. This shows $f$ is single-valued on $\Sigma\setminus(\mathscr{T}\cup\mathcal{Z})$.

Assume $f$ is not continuous at $p\in\Sigma\setminus(\mathscr{T}\cup\mathcal{Z})$. Then $F$ is not continuous at $p$. By definition, there is an $\epsilon>0$ so that for every positive integer $k$, we can find $p_k\in\Sigma\setminus(\mathscr{T}\cup\mathcal{Z})$, where $dist(p,p_k)<1/k$ and $\operatorname{dist}(q,q_k)>\epsilon$, here $q$ and $q_k$ are the unique images of $p$ and $p_k$ under $F$. Let $q'$ be a limit point of $\{q_k\}_{k=1}^{\infty}$. Note that $\operatorname{dist}(q,q')\geq \epsilon$.

Let $d(p)$ be the radius at $p$. Since $d$ is continuous, 

$$\operatorname{dist}(p,q)=d(p)=\lim_{k\to\infty}d(p_k)=\lim_{k\to\infty}\operatorname{dist}(p_k,q_k)=\operatorname{dist}(p,q')$$.

Thus, $q'\in F(p)$. However, $\operatorname{dist}(q,q')\geq \epsilon$ so $F$ is multi-valued at $p$, a contradiction to the single-valuedness of $F$ at $p$. Therefore, $F$ and hence $f$ is continuous on $\Sigma\setminus(\mathscr{T}\cup\mathcal{Z})$. 
\end{proof}

\begin{remark} It turns out that the continuity of $F$ is proved in Lemma 1 of \cite{R1}. We keep the proof for the consistence in notations.
\end{remark}

\begin{lemma} $f$ is a rational function on each connected component of $\Sigma\setminus(\mathscr{T}\cup\mathcal{Z})$.
\end{lemma}

\begin{proof} 

Let $\mathcal{U}$ be a connected component of $\Sigma\setminus(\mathscr{T}\cup\mathcal{Z})$, so $\mathcal{U}\subset{R_s}^{\mathrm{o}}$ for some $s$.

Choose a $p_0\in\mathcal{U}$. Since $f$ is single-valued at $p_0$ by Lemma 4.4, we assume $f(p_0)=\{q_0\}$ . If $q_0\notin \mathscr{C}$, we will show that there exists a rational map $f_{ijk}$, together with a chart map $\Phi:f(\mathcal{U})\to\mathbb{E}$ such that $f={\Phi}^{-1}\circ f_{ijk}\circ\Psi_s$. Otherwise, if $q_0\in \mathscr{C}$, we will show that $f$ is constant on $\mathcal{U}$.

\textit{Case 1. $q_0\notin \mathscr{C}$:}

From Section 4.1, we know that there exists a good triple $\{i,j,k\}$ and a point $q_{0_{ijk}}=f_{ijk}\circ\Psi_s(p_0)$ such that $q_0={Dev_{p_0}}^{-1}(q_{0{ijk}})$. In fact, such triple is unique, otherwise $p_0\in\mathcal{Z}_s\subset\mathcal{Z}$ since it satisfies an equation of \textit{Type 1}.

Now for any $p\in \mathcal{U}$, define $q_{ijk}=f_{ijk}\circ\Psi_s(p)$. We claim that the following inequalities hold throughout $\mathcal{U}$:

\textit{1)} $D(s)_{ijk}(p)>D(s)_{abc}(p)$ for any other good triple $\{a,b,c\}$ at $p$ if exists;

\textit{2)} $D(s)_{ijk}(p)>\norm{\phi_n(p)-C_n(s)}$ for all $n$;

\textit{3)} $D(s)_{ijk}(p)<\norm{\phi_n(p)-q_{ijk}}$ whenever $n\neq i,j,k$ and $[q_{ijk},\phi_n(p)]$ is a $\bigstar_p$-path.

Clearly, at $p_0$, all the three inequalities hold. Since the maps involved in \textit{1)} and \textit{2)} are continuous on $\mathcal{U}$, \textit{1)} and \textit{2)} must hold throughout $\mathcal{U}$, otherwise we can find a $p'\in\mathcal{U}$ satisfying an equation of \textit{Type 1} or \textit{Type 2}, a contradiction. Finally, \textit{3)} means ${\overline{Dev_p}}^{-1}([q_{ijk},\phi_i(p)]),{\overline{Dev_p}}^{-1}([q_{ijk},\phi_j(p)])$ and ${\overline{Dev_p}}^{-1}([q_{ijk},\phi_k(p)])$ are distance minimizers. If \textit{3)} does not hold everywhere on $\mathcal{U}$, then given the continuity of $\phi$ and $f$, there is a $p''\in\mathcal{U}$ such that $\phi(p'')$ is joined to $f(p'')$ by at least four distance minimizers, again $p''$ satisfies a \textit{Type 1} equation, a contradiction. 

From this claim, it follows that at any $p\in\mathcal{U}$, $\{i,j,k\}$ is the unique good triple, so $f(p)$ consists of a single element $q={Dev_p}^{-1}(q_{ijk})$. In addition, we claim $q\notin\mathscr{C}$: if $q_{ijk}=C_n(s)$ for some $n$, then there are two (if $n$ equals one of $i,j,k$) or three (if $n\neq i,j,k$) distance minimizers joining $\phi(p)$ and $C_n$. By definition, $\phi(p)\in\mathscr{T}$, hence $p\in\mathscr{T}$ by symmetry, contradicting $p\in\mathcal{U}$. This observation is important for \textit{Case 2} later.

Now we define a map $\Phi$ on $f(\mathcal{U})$ as follows: if $f(p)=\{q\}$, then $\Phi(q):=Dev_p(q)$. Note that $f={\Phi}^{-1}\circ f_{ijk}\circ\Psi_s$, where $\{i,j,k\}$ is the unique good triple throughout $\mathcal{U}$, and $f_{ijk}$ is a rational function by Lemma 4.3. It remains to show $\Phi$ is a chart map.

Note that $f_{ijk}$ is injective on $\Psi_s(\mathcal{U})$, since $f$ is injective on $\mathcal{U}$ (see Remark 2).

We first show $\Phi$ is injective. Assume $f(p_1)=\{q_1\}$,  $f(p_2)=\{q_2\}$ and $Dev_{p_1}(q_1)=Dev_{p_2}(q_2)$. Then $f_{ijk}\circ\Psi_s(p_1)=f_{ijk}\circ\Psi_s(p_2)$. By injectivity of $f_{ijk}$ and $\Psi_s$, we have $p_1=p_2$, hence $q_1=q_2$.

Then we show $\Phi$ is a local isometry. Fix a $p\in\mathcal{U}$. We have seen that $q$ is not a conical point. In Section 4.1 (the paragraph before Definition 4.2), we have seen that if $q$ is not a conical point, then $q\notin\tau_p=\bigcup\limits_{n=1}^{2N}[\phi(p), C_n]$. Consequently, there is a neighborhood $N_p\subset\mathcal{U}$ of $p$, such that if $p'\in N_p$, then $Dev_p=Dev_{p'}$ on $f(N_p)$. That is, $\Phi=Dev_p$ on $f(N_p)$, so $\Phi$ is a local isometry.

In conclusion, $f={\Phi}^{-1}\circ f_{ijk}\circ\Psi_s$ on $\mathcal{U}$, where $\Psi_s$ and $\Phi$ are both chart maps. 

\textit{Case 2. $q_0\in \mathscr{C}$:}

In \textit{Case 1}, we showed that if $q_0\notin \mathscr{C}$, then $f(\mathcal{U})$ does not contain any conical point. Thus, if $q_0\in \mathscr{C}$, then $f(\mathcal{U})$ consists of conical points only. Since $f$ is continuous on $\mathcal{U}$ by Lemma 4.4, the image of $f$ must be a single conical point throughout $\mathcal{U}$, hence $f$ is constant.

Combined with \textit{Case 1}, we have shown that $f$ is rational on any connected component of $\Sigma\setminus(\mathscr{T}\cup\mathcal{Z})$.
\end{proof}

Note that if $f$ is single-valued at $p\in\mathcal{Z}\setminus\mathscr{T}$, then the coordinates of $p$ satisfies an equation $f_{ijk}(p)=f_{ijl}(p)$ for some mutually distinct $i,j,k,l\in\{1,2,\dots,2N\}$:  

If $f$ is single-valued at $p$, then it is impossible that $C_m$ and $C_n$ are both farthest points from $\phi(p)$, so $p$ can not be a solution of \textit{Type 3} equations only. In addition, if $p$ satisfies \textit{Type 2} equations only, then there are at least two distance minimizers from $\phi(p)$ to $C_n$, so $p\in\mathscr{T}$. Thus, $p$ satisfies an equation of \textit{Type 1}. Since $f$ is single-valued, there are at least four distance minimizers joining $\phi(p)$ and its farthest point, and the conclusion follows.

By Theorem 3 of \cite{R2}, the solution to the equation $f_{ijk}(p)=f_{ijl}(p)$ can not contain an open set. By Theorem 5 of \cite{Z}, $f$ is single-valued outside a $\sigma$-porous set. Therefore, $\mathscr{T}\cup\mathcal{Z}$ is at most one-dimensional.

In Figure 6, we use the computer program to plot the set where $f$ cannot be locally represented by a rational function on the surface of a regular octahedron. The set consists of three types of curves: the multi-valued set (red), the limit set (blue),  and the third type that is neither (green). In the interior of each region bounded by these curves, $f$ is rational.

We compute these special curves as follows: First, we compute the regions $\{R_s\}$. As the second picture of Figure 5 shows, each $R_s$ is a regular triangle. Next, we compute the valid solutions $\mathcal{Z}_s$ for all $s$. We then test a large number of points on each curve. If $p$ has multiple farthest points, it belongs to the multi-valued set and is colored red. In particular, if $p$ is a valid solution to $f_{ijk}(p)=f_{ijl}(p)$, we test whether it is fixed by $f$. If so, it belongs to the limit set and is colored blue. Otherwise, it is colored green.

The program also works for other centrally symmetric octahedra with conical angles all equivalent to $\frac{4\pi}{3}$, and theoretically, the algorithm applies to all centrally symmetric convex polyhedra. However, we expect the program to be much slower with the increasing number of conical points.

\begin{figure}[h]
\centering
\includegraphics[width=0.80\textwidth]{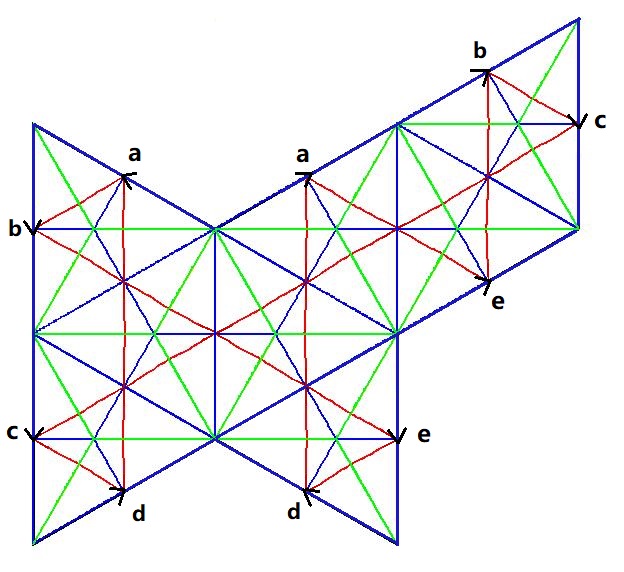}
\caption{In case $\Sigma$ is a regular octahedron, $f$ is represented by a rational function in the interior of each region bounded by the colored curves. To get the regular octahedron, two edges with the same letter should be glued along the indicated direction of the arrows. The blue lines are the limit set of $f$, the red curves are the multi-valued set of $f$, and the green ones are neither. }
\end{figure}

\section{Proof of Theorem 1.3 and 1.4}

In this section we prove Theorem 1.3 and 1.4 based on the following lemma:

\begin{lemma}Suppose $\overline{p}$ is a limit point of some orbit of $f$ through $p$ and $\overline{p}\notin \mathscr{C}$, then there are at least four distance minimizers joining $\overline{p}$ to $\phi(\overline{p})$. 
\end{lemma}

\begin{proof} By Theorem 1.2, $\overline{p}\in f(\overline{p})$, so $\phi(\overline{p})\in F(\overline{p})$. The number of distance minimizers from $\overline{p}$ to $\phi(\overline{p})$ is at least 3 by Lemma 2.4, and is even by symmetry. So this number is at least 4.
\end{proof}

\subsection{Each Orbit of $f$ is a Convergent Sequence}

We begin the proof of Theorem 1.3 with the following lemma.

\begin{lemma} The set of limit points of an orbit of $f$ is finite, hence discrete.
\end{lemma}
\begin{proof} In Section 4.2, we decompose $\Sigma$ into finitely many closed connected regions $R_s$ with isometric embeddings $\Psi_s:{R_s}^{\mathrm{o}}\to\mathbb{E}$. Let $\overline{\Psi_s}$ be the continuous extension of $\Psi_s$ to $R_s$. It suffices to show that there are finitely many limit points of an orbit of $f$ in $R_s$. 

Let $\overline{p}\in R_s$ be a limit point of an orbit $\{p_n\}$ of $f$ through $p$. We assume $\overline{p}\notin \mathscr{C}$, since this doesn't affect finiteness. By Lemma 3.3, $d(\overline{p})=L$, where $L=\lim_{n\to\infty}d(p_n)$. By Lemma 5.1, there are four distinct distance minimizers $g_1, g_2, \phi(g_1)$ and $\phi(g_2)$ joining $\overline{p}$ and $\phi(\overline{p})$. Let $I_i, I_j, I_k$ and $I_l$ be the orientation reversing Euclidean isometries obtained by extending the local ``coordinate representations of $\phi$'' induced by $g_1, g_2, \phi(g_1)$ and $\phi(g_2)$ (see Lemma 4.1 for precise definitions). Since the lengths of these four distance minimizers are all $L$, the coordinates $(x,y)$ of $\overline{\Psi_s}(\overline{p})$ in $\mathbb{R}^2$ must satisfy
\[
\norm{I_i(x,y)-(x,y)}=\norm{I_j(x,y)-(x,y)}=\norm{I_k(x,y)-(x,y)}=\norm{I_l(x,y)-(x,y)}=L 
\tag{$\star$}
\]

Since $\norm{I_i(x,y)-(x,y)}=\norm{I_k(x,y)-(x,y)}$ and $\norm{I_j(x,y)-(x,y)}=\norm{I_l(x,y)-(x,y)}$ by symmetry, this reduces to 

$$\norm{I_i(x,y)-(x,y)}=\norm{I_j(x,y)-(x,y)}=L$$

Consider first the equation $\norm{I_i(x,y)-(x,y)}=L$. As $I_i$ is an orientation-reversing Euclidean isometry, we can write it as a composition of a reflection and a translation. Furthermore, we can choose the reflection axis such that the translation vector is parallel to it. Without loss of generality, suppose this reflection axis is the x-axis. Then we can write $I_i(x,y)=(x+b,-y)$, so $\norm{I_i(x,y)-(x,y)}=L$ if and only if 
$$L^2=\norm{I_i(x,y)-(x,y)}^2=\norm{(x+b,-y)-(x,y)}^2=b^2+4y^2$$

Then $y$ can take at most two values, while $x$ can be arbitrary. So the solution set is empty, one line or a union of two lines, parallel to the axis of reflection. 

Similarly, the solution set to $\norm{I_j(x,y)-(x,y)}=L$ is empty, one line or a union of two lines parallel to the reflection axis of $I_j$. Note that the reflection axis of $I_j$ is not parallel to that of $I_i$, since $I_i{I_j}^{-1}$ is a rotation rather than translation. Therefore, the intersection of the solution sets to these two equations consists of at most $4$ points. 

On $R_k$, there are finitely many choices of $i,j,k$ and $l$, yielding finitely many equations of type $(\star)$. As each equation has at most $4$ solutions, there are finitely many limit points of a given orbit.
\end{proof}

\begin{proof}[Proof of Theorem 1.3] Let $\{p_n\}$ be an orbit of $f$ through $p$. We want to show $\{p_n\}$ is a convergent sequence.

 Let $\overline{p}$ be a limit point of $\{p_n\}$. We claim that for every $\epsilon >0$, there is an integer $N_0$ such that $\operatorname{dist}(p_n, \overline{p})< \frac{\epsilon}{2}$ for all $n>N_0$. This will imply $\operatorname{dist}(p_n,p_m)<\epsilon$ for all $n,m>N_0$, so $\{p_n\}$ is Cauchy.
 
 To prove this claim, let $\epsilon_0$ be such that there are no other limit point than $\overline{p}$ in the $\epsilon_0$-neighborhood of $\overline{p}$. This can be achieved by Lemma 5.2. 
 
 Given any $\epsilon >0$, let $\epsilon '=min\{\epsilon, \epsilon_0\}$. Define a set $$S_{\epsilon '}=\{p_n: \operatorname{dist}(p_n, \overline{p})\geq\frac{\epsilon '}{2}, \operatorname{dist}(p_{n-1}, \overline{p})<\frac{\epsilon '}{2}\}$$
 
 Assume $S_{\epsilon '}$ is an infinite set. Then there is a subsequence $\{p_{n_j}\}$ of $S_{\epsilon '}$ converging to some point $\overline{q}$. Since $\operatorname{dist}(p_{n_j}, \overline{p})\geq \frac{\epsilon '}{2}$, $\overline{q}\neq \overline{p}$. Now $\overline{p}$ and $\overline{q}$ are both limit points of the same orbit, so $d(\overline{p})=d(\overline{q})$ (Lemma 3.3). In addition, $d(\phi(\overline{q}))=d(\overline{q})$ by symmetry. Therefore, 
 $$\operatorname{dist}(\overline{p},\phi(\overline{q}))=\lim_{n_j\to\infty}\operatorname{dist}(p_{n_j-1}, \phi(p_{n_j}))=\lim_{n_j\to\infty}d(p_{n_j-1})=d(\overline{p})=d(\phi(\overline{q}))$$
 where the first equality follows from $\lim_{n_j\to\infty}\{p_{n_j-1}\}=\overline{p}$, since $\operatorname{dist}(p_{n_j-1}, \overline{p})<\frac{\epsilon'}{2}$ for all $n_j$, and $\overline{p}$ is the only limit point in its $\epsilon'$-neighborhood; the second equality comes from $\phi(p_{n_j})\in F(p_{n_j-1})$; the third is due to the continuity of $d$. 
 
 Thus, we conclude that $\phi(\overline{q})\in F(\overline{p})$ and $\overline{p}\in F(\phi(\overline{q}))$, so that $\overline{p}=\overline{q}$ by Lemma 3.2. However, at the beginning of our assumption, we showed that $\overline{q}\neq \overline{p}$. Therefore, $S_{\epsilon '}$ must be a finite set.
 
 Since $S_{\epsilon '}$ is finite, we take $N_0$ to be the maximal number such that $p_{N_0}\in S_{\epsilon '}$. Then $\operatorname{dist}(p_n, \overline{p})<\frac{\epsilon '}{2}$ for all $n>N_0$, since otherwise we will reach a contradiction to the the maximality of $N_0$. Thus the claim is proved.
 \end{proof}

\subsection{Limit Set of $f$ is Contained in at most Quadratic Curves}

Lemma 5.1 implies that the limit set of $f$ is contained in the set $\mathcal{Z}$ of Section 4.2 because they satisfy equations of \textit{Type 1}. In this section, we will show that the limit set is actually a finite union of generalized hyperbolas (that is, including the union of two crossing lines), hence of degree at most 2. 

In Section 5.1, we have seen that if $\overline{p}$ is a  limit point of $f$ in $R_s$, then the coordinates of $\Psi_s(\overline{p})=(\overline{x},\overline{y})$ is a solution to
$$\norm{I_i(x,y)-(x,y)}=\norm{I_j(x,y)-(x,y)}$$
where $I_i$ and $I_j$ are orientation reversing Euclidean isometries induced by two distance minimizers joining $\overline{p}$ and $\phi(\overline{p})$, and ${I_i}^{-1}I_j$ is a rotation. 


\begin{lemma}
By choosing the origin and the real axis appropriately, we can write 

$I_i(z)=e^{-i\alpha}\overline{z}+R_1e^{-i\alpha/2}$

$I_j(z)=e^{i\alpha}\overline{z}+R_2e^{i\alpha/2}$

where $R_1, R_2$ are real numbers. 

\end{lemma}

\begin{proof} Let $A_1$ be the axis of reflection of $I_i$ such that the translation vector is parallel to $A_1$ (same as in the proof of Lemma 5.2). Similarly, let $A_2$ be the axis of reflection of $I_j$ such that the translation vector is parallel to $A_2$. Since $I_i{I_j}^{-1}$ is a rotation rather than translation, $A_1$ intersects $A_2$ at a point, which we take to be the origin $O$. Suppose $A_1$ and $A_2$ form an angle $\alpha$ at $O$ (there are two such angles and we choose one). We take the line bisecting $\alpha$ through $O$ to be the real axis. Then it is not hard to see that $I_i(z)$ and $I_j(z)$ have the desired formula.
\end{proof}

Now the proof of Theorem 1.4 follows from elementary algebra:
\begin{equation*}
\begin{split}
&\norm{I_i(z)-z}=\norm{I_j(z)-z}\\
\iff &(I_i(z)-z)(\overline{I_i(z)}-\overline{z})=(I_j(z)-z)(\overline{I_j(z)}-\overline{z}) \\
\iff &(e^{-i\alpha}\overline{z}+R_1e^{-i\alpha/2}-z)(e^{i\alpha}z+R_1e^{i\alpha/2}-\overline{z})\\
&=(e^{i\alpha}\overline{z}+R_2e^{i\alpha/2}-z)(e^{-i\alpha}z+R_2e^{-i\alpha/2}-\overline{z})\\
\iff &(e^{i\alpha}-e^{-i\alpha})({\overline{z}}^2-z^2)={R_1}^2-{R_2}^2\\
\iff&8\sin(\alpha)xy={R_1}^2-{R_2}^2
\end{split}
\end{equation*}

Since $\alpha\neq 0$ or $\pi$,  this equation characterizes a rectangular hyperbola, which degenerates if and only if $R_1=R_2$.

Therefore, the limit set is contained in a finite union of rectangular hyperbolas and conical points. This proves Theorem 1.4.

\begin{remark}Jo\"el Rouyer suggests the use of these coordinate axis, which simplifies the computation a lot compare to the proof in the last version of this article.
\end{remark}

For the limit set of $f$ on the surface of regular octahedron, we observe degenerated hyperbolas only (see Figure 6, Section 4.2). 

On the last page, we plot the limit sets of $f$ on the surfaces of two octahedra obtained by slightly perturbing the regular octahedron. The second one is also an anti-prism with regular triangular bases, so the limit set displays more symmetry. Each of the two surfaces can be obtained by gluing the two edges with the same letter along indicated directions of the arrows. In both cases, the limit sets are subsets of rectangular hyperbolas, some of which degenerate in the second case.

\begin{figure}[h]
\centering
\includegraphics[width=0.82\textwidth]{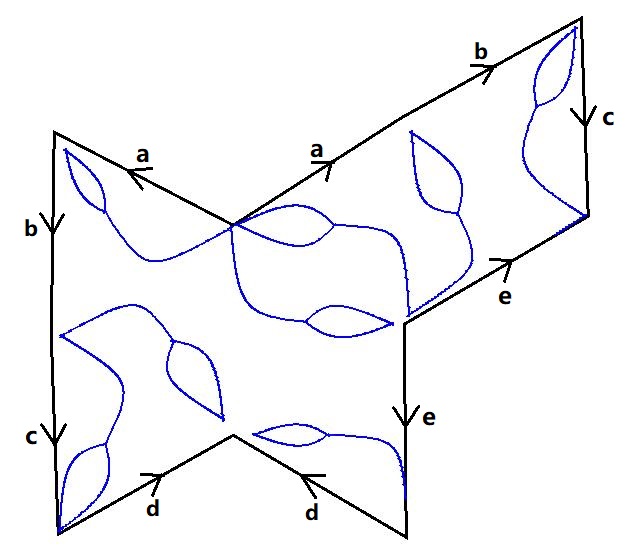}
\includegraphics[width=0.82\textwidth]{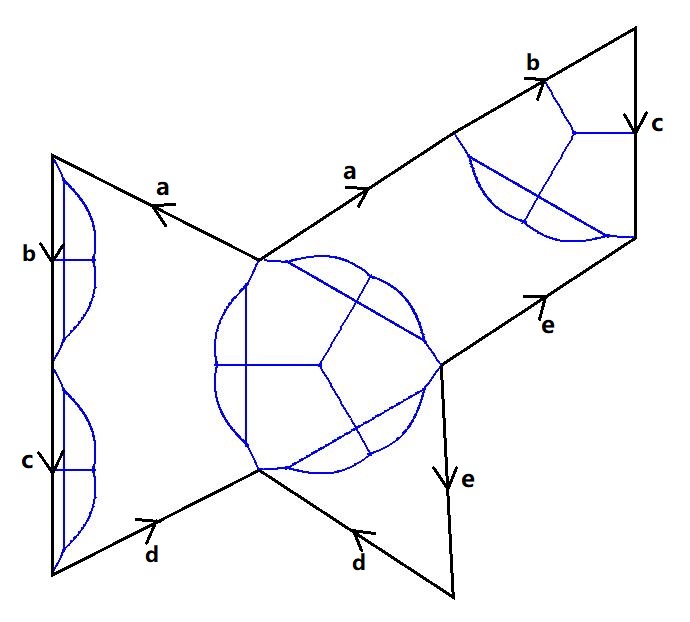}
\end{figure}

%
%

\end{document}